\documentclass[12pt, reqno]{amsart}
\usepackage{times}
\usepackage{amssymb}

\newtheorem{theorem}{Theorem}
\newtheorem{lemma}[theorem]{Lemma}

\newtheorem{corollary}[theorem]{Corollary}

\newcommand{\R}{{\mathbb R}}
\newcommand{\C}{{\mathbb C}}
\newcommand{\inr}{\int_{\R}}
\newcommand{\inc}{\int_{\C}}
\newcommand{\F}{\mathcal{F}}
\newcommand{\Fa}{\F_\alpha}

\begin{document}

\title[Bargmann transform]{The Fourier and Hilbert transforms\\
Under the Bargmann transform}
\author[Dong and Zhu]{Xing-Tang Dong and Kehe Zhu}

\address{Department of Mathematics, Tianjin University, Tianjin 300354, China.}
\email{dongxingtang@163.com}

\address{Department of Mathematics, Shantou University, Shantou, Guangdong 515063, China, and
Department of Mathematics and Statistics, SUNY, Albany, NY 12222, USA.}
\email{kzhu@math.albany.edu}



\thanks{\noindent Dong was supported in part by the National Natural Science Foundation of
China (Grant No. 11201331); Zhu was supported by the National Natural Science Foundation of China
(Grant No. 11371234) and the Project of International Science and Technology Cooperation Innovation
Platform in Universities in Guangdong Province (Grant No. 2014KGJHZ007).}

\begin{abstract}
There is a canonical unitary transformation from $L^2(\R)$ onto the Fock space $F^2$, called the Bargmann
transform. We study the action of the Bargmann transform on several classical integral operators on $L^2(\R)$,
including the fractional Fourier transform, the fractional Hilbert transform, and the wavelet transform.
\end{abstract}

\maketitle

\section{Introduction}

The Fock space $F^2$ is the Hilbert space of all entire functions $f$ on the complex plane $\C$ such that
$$\|f\|^2=\inc |f(z)|^2\,d\lambda(z)<\infty,$$
where 
$$d\lambda(z)=\frac{1}{\pi}e^{-|z|^2}\,dA(z)$$
is the Gaussian measure. Here $dA$ is ordinary area measure. The inner product on $F^2$ is inherited
from $L^2(\C,d\lambda)$. The Fock space is a convenient setting for many problems in functional analysis,
mathematical physics, and engineering. See \cite{Zhu1} for a recent survey of the mathematical theory of
Fock spaces.

Another Hilbert space we consider is $L^2(\R)=L^2(\R,dx)$. We will study the (fractional) Fourier transform,
the (fractional) Hilbert transform, and the wavelet transform as bounded linear operators
on $L^2(\R)$. The books \cite{F1, F2, G, OZK} are excellent sources of information for these operators.

The Bargmann transform $B$ is the operator from $L^2(\R)\rightarrow F^2$ defined by
$$Bf(z)=c\inr f(x)e^{2xz-x^2-(z^2/2)}\,dx,$$
where $c=(2/\pi)^{1/4}$. It is well known that $B$ is a unitary operator from $L^2(\R)$ onto $F^2$.
Furthermore, the inverse of $B$ is also an integral operator, namely,
$$B^{-1}f(x)=c\inc f(z)e^{2x\overline{z}-x^2-(\overline{z}^2/2)}\,d\lambda(z).$$
See \cite{F2, G, Zhu1}.

The Bargmann transform is an old tool in mathematical analysis and mathematical physics. See
\cite{B1,B2,BC,C,F2,G} and references there. In this article we study the action of the Bargmann transform
on several classical integral operators on $L^2(\R)$. The unitarily equivalent version of these operators on
$F^2$ sometimes takes amazingly simple form, sometimes reveals interesting properties, and sometimes
suggests natural new questions. It is our hope that this article will generate some new interest in this classical
area of mathematical analysis.

This paper was completed while the first author visited the Department of Mathematics and Statistics
at the State University of New York at Albany for the 2015-2016 academic year. He wishes to thank
SUNY-Albany for hosting his visit.

\section{Preliminaries}

The standard monomial orthonormal basis for $F^2$ is given by
$$e_n(z)=\sqrt{\frac{1}{n!}}\,z^n, \qquad n\geq0.$$
Thus the reproducing kernel of $F^2$ is
$$K(z,w)=\sum_{n=0}^{\infty}e_n(z)\overline{e_n(w)}=\sum_{n=0}^{\infty}\frac{(z\overline{w})^n}{n!}=e^{z\overline{w}}.$$
The identity
$$f(z)=\inc f(w)K(z,w)\,d\lambda(w)=\inc f(w)e^{z\overline w}\,d\lambda(w),\qquad f\in F^2,z\in\C,$$
is then called the reproducing formula for functions in the Fock space.

To exhibit an orthonormal basis for $L^2(\R)$, recall that for any $n\geq0$ the function
$$H_n(x)=(-1)^ne^{x^2}\frac{d^n}{dx^n}e^{-x^2}$$
is called the $n$th Hermite polynomial. It is well known that the functions
$$h_n(x)=\frac{c}{\sqrt{2^nn!}}e^{-x^2}H_n(\sqrt{2}x), \qquad n\geq0,$$
form an orthonormal basis for $L^2(\R)$, where $c=(2/\pi)^{1/4}$ again. See \cite{W, Zhu1}
for more information about the Hermite functions.

\begin{lemma}\label{1}
For every $n\geq0$ we have $Bh_n=e_n$.
\end{lemma}

\begin{proof}
See Theorem 6.8 of \cite{Zhu1}.
\end{proof}

As a consequence of Lemma~\ref{1} we see that the Bargmann transform is a unitary operator from
$L^2(\R)$ onto $F^2$.

The following elementary result along with some of its close relatives will be used
many times later in the paper.

\begin{lemma}\label{2}
Let $a,b\in\R$ with $a>0$. Then we have
$$\inr e^{-(a+ib)(x+z)^2}\,dx=\frac{\sqrt{\pi}}{\sqrt{a+ib}}$$
for every complex number $z$.
\end{lemma}

\begin{proof}
Write $$I(z)=\inr e^{-(a+ib)(x+z)^2}\,dx,\qquad z\in\C.$$
It is clear that $I(z)$ is an entire function. Moreover, if we write $z=s+it$ for real $s$ and $t$, then
$$\left|e^{-(a+ib)(x+z)^2}\right|=\exp\left(-ax^2\cdot\frac{(x+s)^2-t^2-2tb(x+s)/a}{x^2}\right)\rightarrow0$$
as $x\rightarrow\pm\infty$. Differentiating under the integral sign, we obtain
$$I'(z)=-2(a+ib)\inr (x+z)e^{-(a+ib)(x+z)^2}dx=e^{-(a+ib)(x+z)^2}\bigg|_{-\infty}^{+\infty}=0$$
for any $z\in\C$. It follows that
$$I(z)=I(0)=\inr e^{-(a+ib)x^2}dx=2\int_0^{+\infty} e^{-(a+ib)x^2}dx.$$

Since $a>0$, we can write
$$a+ib=\sqrt{a^2+b^2}\,e^{i\gamma},\qquad \gamma\in\left(-\frac{\pi}{2}, \frac{\pi}{2}\right).$$
If $\gamma=0$, then $b=0$ and the desired result follows immediately. So we assume $\gamma\neq0$.

For large positive $R$ we consider the sector
$$D_R=\left\{z=re^{i\theta}:0< r<R, \min(0,\gamma/2)<\theta<\max(0,\gamma/2)\right\}$$
in the right half-plane. The domain $D_R$ has boundary consisting of three curves: the interval $\overrightarrow{OA}$
on the real axis, the line segment $\overrightarrow{BO}$ on the ray $\theta=\gamma/2$, and the arc $\widehat{AB}$
on the circle $|z|=R$. Since the function $e^{-z^2}$ is analytic on $D_R$, it follows form Cauchy's theorem that
$$\int_{\overrightarrow{OB}}{e^{-z^2}dz}=\int_{\overrightarrow{OA}}{e^{-z^2}dz}+\int_{\widehat{AB}}{e^{-z^2}dz}.$$

Since $\gamma\in(-\pi/2, \pi/2)$, we have
$$\left|\int_{\widehat{AB}}{e^{-z^2}dz}\right|\leq \max_{z\in \widehat{AB}}\left|e^{-z^2}\right|\cdot\frac{|\gamma| R}{2}
=\frac{|\gamma| R e^{-R^2cos\gamma}}{2}\rightarrow 0$$
as $R\rightarrow+\infty$. Therefore,
\begin{align*}
    \int_0^{+\infty} e^{-(a+ib)x^2}dx &=\lim_{R\rightarrow+\infty}\int_0^{R/\sqrt[4]{a^2+b^2}} e^{-(a+ib)x^2}dx\\
    &=\lim_{R\rightarrow+\infty}\frac{1}{\sqrt{a+ib}}\int_{\overrightarrow{OB}}{e^{-z^2}dz}\\
    &=\lim_{R\rightarrow+\infty}\frac{1}{\sqrt{a+ib}}\int_{\overrightarrow{OA}}{e^{-z^2}dz}\\
    &=\frac{1}{\sqrt{a+ib}}\int_0^{+\infty} e^{-x^2}dx=\frac{\sqrt{\pi}}{2\sqrt{a+ib}}.
\end{align*}
This proves the desired result.
\end{proof}

\section{The Fourier transform}

There are several normalizations for the Fourier transform. We define the
Fourier transform by $$\mathcal{F}(f)(x)=\frac{1}{\sqrt{\pi}}\inr e^{-2ixt}f(t)\,dt.$$
It is well known that the Fourier transform is a unitary operator
on $L^2(\R)$, and its inverse is given by
$$\mathcal{F}^{-1}(f)(x)=\frac{1}{\sqrt{\pi}}\inr e^{2ixt}f(t)\,dt.$$

More generally, the notion of fractional Fourier transforms in the form of fractional powers of the
Fourier transform was introduced as early as 1929 (see \cite{W}), and it has become one of the most valuable
and powerful tools in mathematics, quantum mechanics, optics, and signal processing. Thus for any real angle
$\alpha$ we define the $\alpha$-angle fractional Fourier transform by
$$\Fa(f)(x)=\frac{\sqrt{1-i\cot\alpha}}{\sqrt{\pi}}e^{ix^2\cot\alpha}
\inr e^{-2i(xt\csc\alpha-\frac{\cot\alpha}{2}t^2)}f(t)\,dt,$$
where the square root $\sqrt{1-i\cot\alpha}$ is defined such that
$$\arg\sqrt{1-i\cot\alpha}\in(-\pi/2,\pi/2].$$
Obviously, the integral representation above is well defined if $\sin\alpha\neq0$.
We define $\Fa(f)(x)=f(x)$ if $\alpha=0$ and $\Fa(f)(x)=f(-x)$ if $\alpha=\pm\pi$.
This is consistent with the integral representation above in the sense that
$$\lim_{\varepsilon\rightarrow0}\mathcal{F}_{\alpha+\varepsilon}=\Fa$$
for these special values.

Since the trigonometric functions $\csc$ and $\cot$ are periodic with period $2\pi$, it suffices for us to
consider the case $\alpha\in[-\pi,\pi]$. Clearly, when $\alpha=\frac{\pi}{2}$ and $\alpha=-\frac{\pi}{2}$,
the $\alpha$-angle fractional Fourier transform becomes the usual Fourier transform and the inverse Fourier
transform, respectively. See \cite{BM, OZK} for more information about the fractional Fourier transforms.

It is not at all clear from the definition that $\Fa$ is bounded and invertible on $L^2(\R)$. There are also issues
concerning convergence: it is not clear that the integral defining $\Fa(f)$ converges in $L^2(\R)$ for arbitrary
$f\in L^2(\R)$. The situation will change dramatically once we translate $\Fa$ to an operator on the Fock space.
In other words, we will show that, under the Bargmann transform, the operator $\Fa: L^2(\R)\rightarrow L^2(\R)$
is unitarily equivalent to an extremely simple operator on the Fock space $F^2$.

\begin{theorem}\label{3}
The operator
$$T=B\Fa B^{-1}:F^2\rightarrow F^2$$
is given by $Tf(z)=f(e^{-i\alpha}z)$ for all $f\in F^2$.
\end{theorem}

\begin{proof}
For the purpose of applying Fubini's theorem in the calculations below, we assume that $f$ is any polynomial.
Recall that the polynomials are dense in $F^2$, and under the inverse Bargmann transform, they become the
Hermite polynomials times the Gauss function, which have very good integrability properties on the real line.
We still write $c=(2/\pi)^{1/4}$ and also
$$c'=\frac{\sqrt{1-i\cot\alpha}}{\sqrt{\pi}}.$$
It follows from Fubini's theorem that
\begin{align*}
    &\Fa(B^{-1}f)(x)\\
    &=cc'e^{ix^2\cot\alpha}\inr e^{-2i(xt\csc\alpha-\frac{\cot\alpha}{2}t^2)}\,dt
    \inc f(z)e^{2t\overline{z}-t^2-\frac{\overline{z}^2}{2}}\,d\lambda(z)\\
    &=cc'e^{ix^2\cot\alpha}\inc f(z)e^{-\frac{\overline{z}^2}{2}}\,d\lambda(z)
    \inr e^{2t(\overline{z}-ix\csc\alpha)-(1-i\cot\alpha)t^2\,}dt\\
    &=cc'e^{ix^2\cot\alpha}\inc f(z)e^{-\frac{\overline{z}^2}{2}+\frac{(\overline{z}
    -ix\csc\alpha)^2}{1-i\cot\alpha}}\,d\lambda(z)
    \inr e^{-(1-i\cot\alpha)(t-\frac{\overline{z}-ix\csc\alpha}{1-i\cot\alpha})^2}\,dt.
\end{align*}
By Lemma~\ref{2}, we have
\begin{align*}
    \Fa(B^{-1}f)(x)&=ce^{ix^2\cot\alpha}\inc f(z)e^{-\frac{\overline{z}^2}{2}+\frac{(\overline{z}
    -ix\csc\alpha)^2}{1-i\cot\alpha}}\,d\lambda(z)\nonumber\\
    &=ce^{(i\cot\alpha-\frac{\csc^2\alpha}{1-i\cot\alpha})x^2}
    \inc f(w)e^{\frac{1+i\cot\alpha}{2(1-i\cot\alpha)}\overline{w}^2
    -\frac{2i\csc\alpha}{1-i\cot\alpha}x\overline{w}}\,d\lambda(w).
\end{align*}
After simplification, the expression above becomes
\begin{equation}\label{eq1}
    \Fa(B^{-1}f)(x)=ce^{-x^2}\inc f(w)e^{-e^{-2i\alpha}\overline{w}^2/2+2e^{-i\alpha}x\overline{w}}\,d\lambda(w).
\end{equation}
Therefore,
\begin{align*}
    &B\Fa B^{-1}f(z)\\
    &=c^2\inr e^{2xz-\frac{z^2}{2}-2x^2}\,dx
    \inc f(w)e^{-e^{-2i\alpha}\overline{w}^2/2+2e^{-i\alpha}x\overline{w}}\,d\lambda(w)\\
    &=c^2e^{-\frac{z^2}{2}}\inc f(w)e^{-e^{-2i\alpha}\overline{w}^2/2}\,d\lambda(w)
    \inr e^{2x(z+e^{-i\alpha}\overline{w})-2x^2}\,dx\\
    &=c^2e^{-\frac{z^2}{2}}\inc f(w)e^{-\frac{e^{-2i\alpha}\overline{w}^2}{2}+
    \frac{(z+e^{-i\alpha}\overline{w})^2}{2}}\,d\lambda(w)\inr e^{-2\left(x-\frac{z+e^{-i\alpha}\overline{w}}{2}\right)^2}\,dx.
\end{align*}
It follows from Lemma~\ref{2} again that
$$\inr e^{-2\left(x-\frac{z+e^{-i\alpha}\overline{w}}{2}\right)^2}\,dx
=\sqrt{\frac{\pi}2}=\frac{1}{c^2}.$$
Therefore, we have
$$B\Fa B^{-1}f(z)=\inc f(w)e^{e^{-i\alpha}z\overline{w}}\,d\lambda(w).$$
This together with the reproducing formula for functions in $F^2$ gives
$$B\Fa B^{-1}f(z)=f(e^{-i\alpha}z),$$
which completes the proof of the theorem.
\end{proof}

Since the classical Fourier transform $\F$ is just $\Fa$ with $\alpha=\frac{\pi}{2}$, we see that the Fourier
transform on $L^2(\R)$ is unitarily equivalent to the operator $f(z)\rightarrow f(-iz)$ on $F^2$. This is
well known to experts in the field and can be found in \cite{B1, B2}.

As a consequence of Theorem~\ref{3}, we immediately derive a number of basic properties for the fractional
Fourier transform $\Fa$. In particular, we obtain an alternative proof of the fractional Fourier inversion formula
and the associated Plancherel's formula.

\begin{corollary}\label{4}
The $\alpha$-angle fractional Fourier transform $\Fa$ is a unitary operator on $L^2(\R)$, so that
$$\inr|\Fa(f)|^2\,dx=\inr|f|^2\,dx$$
for all $f\in L^2(\R)$. Furthermore, $\left(\Fa\right)^{-1}=\mathcal{F}_{-\alpha}$.
\end{corollary}

\begin{proof}
This is obvious from the unitarily equivalent form of $\Fa$ on the Fock space.
\end{proof}

The following result is also clear from our new representation of the fractional Fourier
transform on the Fock space, because an entire function uniquely determines
its Taylor coefficients.

\begin{corollary}\label{5}
For each $n\geq0$ the Hermite function $h_n$ is an eigenvector of the fractional
Fourier transform $\Fa$ and the corresponding eigenvalue is $e^{-in\alpha}$.
\end{corollary}

The corollary above actually gives the complete spectral picture for the unitary operator
$\Fa$ on $L^2(\R)$. More specifically, if $\alpha$ is a rational multiple of $\pi$, then the
spectrum of $\Fa:L^2(\R)\to L^2(\R)$ is given by
$$\sigma(\Fa)=\left\{e^{-in\alpha}:n=0,1,2,\cdots\right\},$$
which consists of only finitely many points on the unit circle. If $\alpha$ is an irrational
multiple of $\pi$, then $\sigma(\Fa)$ is the entire unit circle. In particular, if we specialize
to the case $\alpha=\pi/2$, we obtain the following spectral decomposition for the classical
Fourier transform $\F$.

\begin{corollary}
For each $0\leq k\leq3$ let $X_k$ denote the closed subspace of $L^2(\R)$ spanned by the Hermite functions
$h_{k+4m}$, $m\geq0$, and let $P_k:L^2(\R)\rightarrow X_k$ be the orthogonal projection. Then
$$L^2(\R)=\bigoplus_{k=0}^3X_k,$$
and the corresponding spectral decomposition for the unitary operator $\F:L^2(\R)\rightarrow L^2(\R)$
is given by $\F=P_0-iP_1-P_2+iP_3$. In particular, the fixed points of the Fourier transform $\F$ are exactly
functions of the form
$$f(x)=\sum_{n=0}^{\infty}c_nh_{4n}(x), \qquad \{c_n\}\in l^2.$$
\end{corollary}

Each $X_k$ is nothing but the eigenspace of the Fourier transform $\F$ corresponding to the eigenvalue $(-i)^k$.
If we write $\alpha=a \pi/2$, then the eigenvalues of $\Fa$ and $\F$ for eigenfunctions $h_n$ have the following
relation:
$$e^{-in\alpha}=\left(e^{-in\frac{\pi}{2}}\right)^a, \qquad n\geq0.$$
Since the functions $h_n(x)$ form an orthonormal basis for $L^2(\R)$, it is then clear that we can naturally think
of the fractional Fourier transform $\Fa$ as a fractional power of the Fourier transform $\F$.

\section{The Hilbert transform}

The Hilbert transform is the singular integral operator defined by
$$Hf(x)=\frac{1}{\pi}\inr \frac{f(t)dt}{x-t},$$
where the improper integral is taken in the sense of principal value. It is well known that the Fourier transform of
$Hf(x)$ is $-i\; \textrm{sgn} (x)\mathcal{F}(f)(x)$. Thus $H$ is a unitary operator on $L^2(\R)$ and
$$Hf(x)=\mathcal{F}^{-1}\left[-i\; \textrm{sgn} (x)\mathcal{F}(f)(x)\right].$$
Observe that
$$-i\; \textrm{sgn} (x)=e^{-i\pi/2}h(x)+e^{i\pi/2}h(-x),$$
where $h(x)$ is the Heaviside step function: $h(x)=1$ for $x\geq0$ and $h(x)=0$ for $x<0$.
So it is natural to define fractional Hilbert transforms as follows.
$$H^\alpha_{\phi}f(x)=\Fa^{-1}\left[\left(e^{-i\phi}h(x)+e^{i\phi}h(-x)\right)\Fa(f)(x)\right],$$
where $\alpha$ and $\phi$ are real parameters. For $\phi=\pi/2$ and $\alpha=\pi/2$ we recover the classical
Hilbert transform.

Since each fractional Fourier transform $\Fa$ is a unitary operator on $L^2(\R)$, the fractional Hilbert transform
$H^\alpha_\phi$ is also a unitary operator on $L^2(\R)$. The operators $H^\alpha_\phi$ play an important role
in optics and signal processing. See \cite{BM, LMZ} for more information about fractional Hilbert transforms.

In order to identify the operator on the Fock space that corresponds to the fractional Hilbert transform
$H^\alpha_\phi: L^2(\R)\to L^2(\R)$, we need the entire function
\begin{align*}
A_\phi(z)&=\frac{\sqrt{\pi}}{2}\left(e^{-i\phi}+e^{i\phi}\right)+\int_0^z \left(e^{-i\phi}-e^{i\phi}\right)e^{-u^2}\,du\\
&=\sqrt\pi\cos\phi-2i\sin\phi\int_0^ze^{-u^2}\,du.
\end{align*}

\begin{theorem}\label{7}
The operator
$$T=BH^\alpha_{\phi}  B^{-1}:F^2\rightarrow F^2$$
is given by
$$
Tf(z)=\frac{1}{\sqrt{\pi}}\inc f(w)e^{z\overline{w}}A_\phi\left(\frac{e^{i\alpha}z
+e^{-i\alpha}\overline{w}}{\sqrt{2}}\right)\,d\lambda(w)$$
for all $f\in F^2$.
\end{theorem}

\begin{proof}
For $c=(2/\pi)^{1/4}$ again and $f\in F^2$ (we may start out with a polynomial
in order to justify the use of Fubini's theorem), it follows from \eqref{eq1} that
\begin{align*}
&B\left[\left(e^{-i\phi}h(x)+e^{i\phi}h(-x)\right)\Fa\left(B^{-1}f\right)(x)\right]\\
    &=c^2\inr e^{2xz-\frac{z^2}{2}-2x^2}\left(e^{-i\phi}h(x)+e^{i\phi}h(-x)\right)\,dx\\
    &\qquad\inc f(w)e^{-\frac{(e^{-i\alpha}\overline{w})^2}{2}+2e^{-i\alpha}x\overline{w}}\,d\lambda(w)\\
    &=c^2\!\!\inc f(w)e^{e^{-i\alpha} z\overline{w}}d\lambda(w)\!\!\inr e^{-2
    \left(x-\frac{z+e^{-i\alpha}\overline{w}}{2}\right)^2}\left(e^{-i\phi}h(x)+e^{i\phi}h(-x)\right)dx.
\end{align*}
By Theorem\ref{3}, we have
$$B\Fa^{-1}B^{-1}f(z)=f(e^{i\alpha}z).$$
Combining this with the integral formula above, we obtain
\begin{align*}
&Tf(z)\\
&=B\Fa^{-1}B^{-1} B\left[\left(e^{-i\phi}h(x)+e^{i\phi}h(-x)\right)\Fa\left(B^{-1}f\right)(x)\right]\\
    &=c^2\!\!\inc f(w)e^{z\overline{w}}\,d\lambda(w)\!\!\inr e^{-2
    \left(x-\frac{e^{i\alpha}z+e^{-i\alpha}\overline{w}}{2}\right)^2}\left(e^{-i\phi}h(x)+e^{i\phi}h(-x)\right)\,dx\\
    &=\frac{1}{\sqrt{\pi}}\!\inc f(w)e^{z\overline{w}}\,d\lambda(w)\!\!\inr e^{-
    \left(x-\frac{e^{i\alpha} z+e^{-i\alpha}\overline{w}}{\sqrt{2}}\right)^2}\left(e^{-i\phi}h(x)+e^{i\phi}h(-x)\right)\,dx.
\end{align*}
Consider the entire function
$$J(z)=\inr e^{-\left(x+z\right)^2}\left(e^{-i\phi}h(x)+e^{i\phi}h(-x)\right)dx.$$
By the same argument used in the proof of Lemma~\ref{2}, we have
\begin{align*}
    J'(z)&=\frac{d}{dz}\left[e^{-i\phi}\int_0^{+\infty} e^{-\left(x+z\right)^2}\,dx
    +e^{i\phi}\int_{-\infty}^0 e^{-\left(x+z\right)^2}\,dx\right]\\
    &=-2e^{-i\phi}\int_0^{+\infty}\left(x+z\right)e^{-\left(x+z\right)^2}\,dx
    -2e^{i\phi}\int_{-\infty}^0\left(x+z\right)e^{-\left(x+z\right)^2}\,dx\\
    &=e^{-i\phi}e^{-\left(x+z\right)^2}\bigg|_{0}^{+\infty}+e^{i\phi}e^{-\left(x+z\right)^2}\bigg|_{-\infty}^{0}\\
    &=-e^{-i\phi}e^{-z^2}+e^{i\phi}e^{-z^2},
\end{align*}
and
$$J(0)=e^{-i\phi}\int_0^{+\infty} e^{-x^2}\,dx+e^{i\phi}\int_{-\infty}^0 e^{-x^2}\,dx
=\frac{\sqrt{\pi}}{2}\left(e^{-i\phi}+e^{i\phi}\right).$$
Therefore,
$$J(z)=\int_0^z \left(-e^{-i\phi}+e^{i\phi}\right)e^{-u^2}du+\frac{\sqrt{\pi}}{2}\left(e^{-i\phi}+e^{i\phi}\right)=A_\phi(-z).$$
Thus
\begin{align*}
Tf(z)&=\frac{1}{\sqrt{\pi}}\inc f(w)e^{z\overline{w}}J\left(-\frac{e^{i\alpha}z
+e^{-i\alpha}\overline{w}}{\sqrt{2}}\right)\,d\lambda(w)\\
&=\frac{1}{\sqrt{\pi}}\inc f(w)e^{z\overline{w}}A_\phi\left(\frac{e^{i\alpha}z
+e^{-i\alpha}\overline{w}}{\sqrt{2}}\right)\,d\lambda(w).
\end{align*}
This proves the desired result.
\end{proof}

The operator $T$ above appears to be a very interesting integral operator on the Fock space, although
we are unable to verify directly that $T$ is bounded on $F^2$.  A natural problem here is to study the
spectral properties of the integral operator $T$ above on $F^2$, or equivalently, the fractional Hilbert
transform as an operator on $L^2(\R)$. Not much seems to be known,
which is in sharp contrast to the case of the fractional Fourier transform.

\begin{corollary}\label{8}
For any real $\alpha$ and $\phi$ we have
$$H^\alpha_\phi=(\cos\phi)I+(\sin\phi)H^\alpha_{\pi/2}$$
and
$$H^{\pi/2}_\phi=(\cos\phi)I+(\sin\phi)H,$$
where $I$ is the identity operator.
\end{corollary}

\begin{proof}
The desired results follow easily from the decomposition
$$A_\phi(z)=\sqrt\pi\,\cos\phi+(\sin\phi)A_{\pi/2}(z)$$
and the reproducing formula for functions in the Fock space.
\end{proof}

As a generalization of the Hilbert transform, the second formula in the corollary above 
was actually used as one of the definitions for the fractional Hilbert transform in \cite{LMZ}.

Corollary \ref{8} also shows clearly how $H^\alpha_\phi$ depends on $\phi$. It appears that introducing an
extra angle $\phi$ (other than $\pi/2$) does not really produce something new. In particular, for any
$\phi\in(0,\pi/2)$ the operator $H^{\pi/2}_\phi$ is a convex combination of the identity operator $I$ and the
ordinary Hilbert transform $H$.

As a consequence of Theorem~\ref{7} we also obtain the following result which can be found
in \cite[Theorem 1]{Zhu2}.

\begin{corollary}\label{9}
Suppose $A(z)$ is the anti-derivative of $e^{z^2}$ with $A(0)=0$ and $S=BH B^{-1}$. Then
$$Sf(z)=\frac{2}{\sqrt{\pi}}\inc f(w)e^{z\overline{w}}A\left(\frac{z-\overline{w}}{\sqrt{2}}\right)\,d\lambda(w)$$
for all $f\in F^2$ and $z\in\C$.
\end{corollary}

\begin{proof}
Since $H=H^\alpha_{\phi}$ with $\phi=\frac{\pi}{2}$ and $\alpha=\frac{\pi}{2}$, and
$$A_{\pi/2}(z)=-2i\int_0^ze^{-u^2}\,du=2\int_0^{-iz} e^{u^2}\,du=2A(-iz),$$
it follows from Theorem~\ref{7} that
\begin{align*}
Sf(z)&=\frac{1}{\sqrt{\pi}}\inc f(w)e^{z\overline{w}}A_{\pi/2}\left(\frac{iz-i\overline{w}}{\sqrt{2}}\right)\,d\lambda(w)\\
&=\frac{2}{\sqrt{\pi}}\inc f(w)e^{z\overline{w}}A\left(\frac{z-\overline{w}}{\sqrt{2}}\right)\,d\lambda(w),
\end{align*}
as desired.
\end{proof}

\section{Singular integral operators on $F^2$}

Motivated by Corollary~\ref{9}, the second author in \cite{Zhu2} proposed to study the boundedness of more
general ``singular integral operators" on $F^2$ of the form
$$
S_{\varphi}f(z)=\inc f(w)e^{z\overline{w}}\varphi\left(z-\overline{w}\right)d\lambda(w),
$$
where $\varphi$ is any function in $F^2$. In view of Theorem~\ref{7}, it is also natural for us to consider
integral operators on $F^2$ of the form
$$S_{\varphi}^{\alpha}f(z)=\inc f(w)e^{z\overline{w}}\varphi\left(e^{i\alpha}z
-e^{-i\alpha}\overline{w}\right)\,d\lambda(w),$$
where $\alpha$ is a real parameter. Note that we have made an adjustment here so that $S^\alpha_\varphi$
becomes $S_\varphi$ when $\alpha=0$.

The most fundamental problem in this direction is to characterize those $\varphi\in F^2$ such that
$S_{\varphi}$ is bounded on $F^2$. The extra parameter $\alpha$ in $S_{\varphi}^{\alpha}$ does not
yield additional difficulty, because $S_\varphi^\alpha$ is unitarily equivalent to $S_\varphi$ via the Fock
space version of the fractional Fourier transform $\Fa$. In fact, if we write $U_\alpha=B\Fa B^{-1}$,
so that $U_\alpha f(z)=f(e^{-i\alpha}z)$ for $f\in F^2$, then $U_\alpha$ is a unitary operator on $F^2$ and
by the simple change of variables $w=e^{-i\alpha}u$,
$$S^\alpha_\varphi f(z)=\inc f(e^{-i\alpha}u)e^{ze^{i\alpha}\overline u}\varphi(e^{i\alpha}z-\overline u)
\,d\lambda(u)=U_{-\alpha}S_\varphi U_\alpha f(z).$$
Thus $S^\alpha_\varphi$ and $S_\varphi$ are unitarily equivalent as operators on $F^2$.

Several non-trivial examples of $S_\varphi$ were considered in \cite{Zhu2}. It is still an open question to
characterize the boundedness of $S_\varphi$ in terms of properties of $\varphi$. In this section we
will construct additional examples of bounded operators $S_\varphi$ on $F^2$. Our construction is based
on the wavelet transform as an operator on $L^2(\R)$.

Fix $g\in L^2(\R)$ and $s\in \R$ with $s\neq0$, the continuous wavelet transform of $f$ with respect to
the ``wavelet'' $g$ is defined to be
$$W_g(f)(x)=\frac{1}{\sqrt{|s|\pi}}\inr f(t)g(s^{-1}(t-x))\,dt, \qquad f\in L^2(\R).$$
It is well known that the Fourier transform of $W_g(f)$ is given by
$$|s|^{1/2}\overline{\mathcal{F}(g)(sx)}\mathcal{F}\left(f\right)(x).$$ See \cite{G} for more information about
the wavelet transform. So the corresponding operator on the Fock space is defined by
$$BW_gB^{-1}f(z)=B\mathcal{F}^{-1}\left[\overline{D_{s}\mathcal{F}(g)}\mathcal{F}\left(B^{-1}f\right)\right](z),
\qquad f\in F^2,$$
where $D_s: L^2(\R)\rightarrow L^2(\R)$ is the dilation operator
defined by $$D_sg(x)=|s|^{1/2}g(sx).$$

\begin{lemma}\label{10}
The operator
$$T=BW_g B^{-1}:F^2\rightarrow F^2$$
is given by
$$Tf(z)=\sqrt{\frac{|s|}{\pi}}\inc f(w)e^{z\overline{w}}\,d\lambda(w)
\inr g(t)e^{-\frac{s^2}{2}t^2-st(z-\overline{w})}\,dt$$
for all $f\in F^2$.
\end{lemma}

\begin{proof}
For $c=(2/\pi)^{1/4}$ and any polynomial $f$ in $F^2$, it follows from \eqref{eq1} that
$$\F\left(B^{-1}f\right)(x)=ce^{-x^2}\inc f(w)e^{\frac{\overline{w}^2}{2}-2ix\overline{w}}\,d\lambda(w).$$
Therefore,
\begin{align*}
&B\left[\overline{D_{s}\F(g)}\F\left(B^{-1}f\right)\right](z)\\
&=\frac{c^2\sqrt{|s|}}{\sqrt{\pi}}\inr e^{2xz-\frac{z^2}{2}-2x^2}\,dx
\inc f(w)e^{\frac{\overline{w}^2}{2}-2ix\overline{w}}d\lambda(w)\inr g(t)e^{2ixst}\,dt\\
&=\frac{c^2\sqrt{|s|}}{\sqrt{\pi}}\inc f(w)d\lambda(w)\inr g(t)\,dt
\inr e^{2x(z-i\overline{w}+ist)-2x^2-\frac{z^2}{2}+\frac{\overline{w}^2}{2}}\,dx\\
&=\frac{c^2\sqrt{|s|}}{\sqrt{\pi}}\!\!\inc\!f(w)e^{-iz\overline{w}}\,d\lambda(w)
\!\!\inr\!g(t)e^{-\frac{s^2}{2}t^2+ist(z-i\overline{w})}dt\!\!\inr\!e^{-2\left(x-\frac{z-i\overline{w}+ist}{2}\right)^2}\!dx\\
&=\frac{\sqrt{|s|}}{\sqrt{\pi}}\inc f(w)e^{-iz\overline{w}}\,d\lambda(w)
\inr g(t)e^{-\frac{s^2}{2}t^2+ist(z-i\overline{w})}\,dt.
\end{align*}
Recall from Theorem~\ref{3} that
$$B\F^{-1}B^{-1}f(z)=f(iz).$$
Thus we have
\begin{align*}
T(f)(z)&=B\F^{-1}B^{-1} B\left[\overline{D_{s}\F(g)}\F\left(B^{-1}f\right)\right](z)\\
&=\frac{\sqrt{|s|}}{\sqrt{\pi}}\inc f(w)e^{z\overline{w}}\,d\lambda(w)
\inr g(t)e^{-\frac{s^2}{2}t^2-st(z-\overline{w})}\,dt.
\end{align*}
This proves the desired result.
\end{proof}

As a consequence of Lemma~\ref{10} we obtain a class of functions $\varphi\in F^2$ such that the associated
singular integral operator $S_\varphi$ is bounded on $F^2$.

\begin{corollary}\label{11}
For any $g\in L^1(\R)\cap L^2(\R)$ and $s\neq0$ the function
$$\varphi(z)=\sqrt{\frac{|s|}{\pi}}\inr g(t)e^{-\frac{s^2}{2}t^2-t sz}\,dt$$
belongs to the Fock space $F^2$ and the corresponding operator $S_{\varphi}$ is bounded on $F^2$.
\end{corollary}

\begin{proof}
By an obvious change of variables, there is a constant $C_1$ such that
$$\varphi(z)=C_1\inr g\left(-\frac{\sqrt2}s\,x\right)e^{-x^2+\sqrt2 xz}\,dx.$$
The function
$$h(x)=g\left(-\frac{\sqrt2}s\,x\right)$$
is still in $L^2(\R)$ and there is another constant $C_2$ such that
$$\varphi(z)=C_2e^{\frac14z^2}Bh(z/\sqrt2),$$
where $B$ is the Bargmann transform. Since the function $Bh$ is in $F^2$, we have
\begin{align*}
\inc|\varphi(z)|^2\,d\lambda(z)&=|C_2|^2\inc\left|e^{\frac14z^2}Bh\left(\frac z{\sqrt2}\right)\right|^2\,d\lambda(z)\\
&=2|C_2|^2\inc|Bh(w)|^2|e^{w^2}|e^{-|w|^2}\,d\lambda(w)\\
&\le2|C_2|^2\inc|Bh(w)|^2\,d\lambda(w)<\infty.
\end{align*}
This shows that the function $\varphi$ is in $F^2$.

By Lemma \ref{10}, the wavelet transform $W_g$ is unitarily equivalent to $S_{\varphi}$. Since $g\in L^1(\R)$, 
it is clear from Young's inequality for the convolution operator that $W_g$ is bounded on $L^2(\R)$. This proves
the desired result.
\end{proof}

We will consider two special cases of $\varphi\in F^2$ that arise from Corollary~\ref{11}.

First, for any nonnegative integer $n$, let $g_n$ be the monomial $x^n$ times the Gauss function $e^{-x^2}$, and
let $\varphi_n$ denote the corresponding function from Corollary~\ref{11}. We have $g_0(x)=e^{-x^2}$ and
$g_1(x)=xe^{-x^2}$. It follows that
$$\varphi_0(z)=\sqrt{\frac{|s|}{\pi}}\inr e^{-\frac{(s^2+2)}{2}t^2-t sz}\,dt
=\sqrt{\frac{2|s|}{s^2+2}}e^{\frac{s^2}{2(s^2+2)}z^2}$$
and
$$\varphi_1(z)=\sqrt{\frac{|s|}{\pi}}\inr te^{-\frac{(s^2+2)}{2}t^2-t sz}\,dt
=-\frac{s z}{s^2+2}\varphi_0(z).$$
For $g_n(x)=x^ne^{-x^2}$ with $n\geq 2$ we have
\begin{align*}
\varphi_n(z)&=\sqrt{\frac{|s|}{\pi}}\inr t^n e^{-\frac{(s^2+2)}{2}t^2-t sz}\,dt\\
&=\sqrt{\frac{|s|}{\pi}}\inr \left[\left(t+\frac{s z}{s^2+2}\right)t^{n-1}-\frac{s z}{s^2+2}t^{n-1}\right]
e^{-\frac{(s^2+2)}{2}t^2-t sz}\,dt\\
&=-\frac{s z}{s^2+2}\varphi_{n-1}(z)+\frac{n-1}{s^2+2}\varphi_{n-2}(z).
\end{align*}
Therefore, by induction on $n$, we obtain
$$\varphi_n(z)=\left(a_nz^n+\cdots+a_1z+a_0\right)e^{\frac{s^2}{2(s^2+2)}z^2},$$
where
$$a_n=\sqrt{\frac{2|s|}{s^2+2}}\frac{(-1)^ns^n}{\left(s^2+2\right)^n}\neq0.$$

Second, we consider $s=1$ and
$$g(x)=e^{-\frac{\varepsilon}{2} x^2+bx},\qquad\varepsilon>0,b\in\R.$$
It is easy to check that the corresponding function $\varphi$ in Corollary~\ref{11} is given by
$$\varphi(z)=\frac{1}{\sqrt{\pi}}\inr e^{-\frac{1+\varepsilon}{2}t^2-t(z-b)}\,dt
=\sqrt{\frac{2}{1+\varepsilon}}e^{\frac{1}{2(1+\varepsilon)}(z-b)^2}.$$
It follows that the operator $S_{\varphi}$ induced by $\varphi(z)=e^{a(z-b)^2}\in F^2$, where $0<a<1/2$ and
$b\in\R$, is bounded on $F^2$. Furthermore, the range for $a$ is best possible. The case $b=0$ was
proved in \cite{Zhu2}.

It is natural to wonder whether Corollary~\ref{11} might suggest a characterization for the boundedness
of $S_\varphi$ on $F^2$, namely, is it true that $S_\varphi$ is bounded on $F^2$ if and only if $\varphi$
came from a function $g\in L^1(\R)\cap L^2(\R)$ via the integral transform in Corollary~\ref{11}. Unfortunately, the
answer is negative. In fact, if we take $s=-1$ and $g(x)=\frac{1}{\sqrt{\pi}x}$, then $g$ is not in $L^1(\R)\cap L^2(\R)$,
$$W_g(f)(x)=\frac{1}{\pi}\inr \frac{f(t)}{x-t}\,dt=H(f)(x),$$
and the corresponding function $\varphi$ is given by
$$\varphi(z)=\frac{1}{\pi}\inr \frac{e^{-\frac{t^2}{2}+tz}}{t}\,dt=\frac{1}{\pi}\inr \frac{e^{-t^2+\sqrt{2}tz}}{t}\,dt,$$
where the integral above is a ``principle value" integral. We can rewrite this PV-integral in the form of an
ordinary integral as follows:
$$\varphi(z)=\frac{1}{\pi}\inr \frac{e^{-t^2}(e^{\sqrt{2}tz}-1)}{t}\,dt.$$
The singularity at $t=0$ and the singularity at infinity are both gone. Thus
we can differentiate inside the integral sign to get
$$\varphi'(z)=\frac{\sqrt{2}}{\pi}\inr e^{-t^2+\sqrt{2}tz}\,dt=
\frac{\sqrt{2}}{\pi}e^{\frac{z^2}{2}}\inr e^{-(t-z/\sqrt{2})^2}\,dt=\sqrt{\frac{2}{\pi}}e^{\frac{z^2}{2}}.$$
Recall that $A(z)$ is the anti-derivative of $e^{z^2}$ with $A(0)=0$. Since $\varphi(0)=0$, we must have
$$\varphi(z)=\frac{2}{\sqrt{\pi}}A\left(\frac{z}{\sqrt{2}}\right).$$
It follows from the Taylor expansion of $\varphi$, the standard orthonormal basis of $F^2$, and Stirling's
formula that $\varphi\in F^2$. Moreover, it follows from Lemma~\ref{10} that $BHB^{-1}=S_\varphi$.
Recall that the Hilbert transform $H$ is a unitary operator on $L^2(\R)$,
so $S_\varphi$ is also bounded on $F^2$. This gives an alternative proof of Corollary~\ref{9} and shows that
the functions $\varphi$ in Corollary~\ref{11} generated by $g\in L^1(\R)\cap L^2(\R)$ cannot characterize all bounded
operators $S_\varphi$ on $F^2$.


\begin{thebibliography}{99}
\bibitem{B1} V. Bargmann, On a Hilbert space of analytic functions and an associated integral
transform I, \textit{Comm. Pure Appl. Math}. \textbf{14} (1961), 187-214.
\bibitem{B2} V. Bargmann, On a Hilbert space of analytic functions and an associated integral
transform II, \textit{Comm. Pure Appl. Math}. \textbf{20} (1967), 1-101.
\bibitem{BC} C. Berger and L. Coburn, Heat flow and Berezin-Toeplitz estimates, \textit{Amer. J. Math.}
\textbf{116} (1994), 563-590.
\bibitem{BM} A. Bultheel and H. Mart\'inez, Recent developments in the theory of the
fractional Fourier transforms and linear canonical transforms. \textit{Bull. Belg. Math. Soc. Simon Stevin}
\textbf{13} (2007), 971-1005.
\bibitem{C}L. Coburn, The Bargmann isometry and Gabor-Daubechies wavelet localization operators,
in \textit{Systems, Approximation, Singular Integral Operators, and Related Topics},
(Bordeaux 2000), 169-178; \textit{Oper. Theory Adv. Appl.} \textbf{129}, Birkhauser, Basel, 2001.
\bibitem{F1} G. Folland, \textit{Fourier Analysis and Its Applications}, Brooks/Cole Publishing Company, 1992.
\bibitem{F2} G. Folland, \textit{Harmonic Analysis in Phase Space}, Ann. Math. Studies \textbf{122}, Princeton
University Press, 1989.
\bibitem{G} K. Gr\"{o}chenig, \textit{Foundations of Time-Frequency Analysis}, Birkh\"{a}user, Boston, 2001.
\bibitem{LMZ} A.W. Lohmann, D. Mendlovic, and Z. Zalevsky, Fractional Hilbert transform. \textit{Optics
Letters}, \textbf{21} (1996), 281-283.
\bibitem{OZK} H.M. Ozaktas, Z. Zalevsky, and M.A. Kutay, \textit{The Fractional Fourier Transform: with Applications in Optics and Signal Processing}, Wiley, Chichester, 2001.
\bibitem{W} N. Wiener, Hermitian polynomials and Fourier analysis. \textit{J. Math. Phys.}, \textbf{8} (1929), 70-73.
\bibitem{Zhu1} K. Zhu, \textit{Analysis on Fock Spaces}, Springer, New York, 2012.
\bibitem{Zhu2} K. Zhu, Singular integral operators on the Fock space, \textit{Integr. Equat. Oper. Theory},
\textbf{81} (2015), 451-454.
\bibitem{Zhu3} K. Zhu, Towards a dictionary for the Bargmann transform (arXiv:1506.06326, June 2015),
manuscript not for publication.
\end{thebibliography}
\end{document}